\documentclass[12pt,a4paper]{article}
\usepackage[utf8]{inputenc}
\usepackage{amsmath}
\usepackage{amsfonts}
\usepackage{amssymb}
\usepackage{graphicx}
\usepackage{mathabx}

    \usepackage{pdfsync}

\newtheorem{lemma}{Lemma}[section]

\newtheorem{proposition}{Proposition}[section]

\newenvironment{proof}[1][Proof]{\textbf{#1.} }{\ \rule{0.5em}{0.5em}}
\author{Yves Le Jan}
\title{Markov Loops, Free Field and Eulerian Networks} 
\begin{document}
\maketitle

\begin{abstract}
We investigate the relations between the Poissonnian loop ensemble arising in the construction of random spanning trees, the free field, and random Eulerian networks.
\end{abstract}


\section{The loop ensemble and the free field}
We adopt the framework described in \cite{stfl}. Given a graph $\mathcal{G}=(X,E)$, a set of non negative conductances $C_{x,y}=C_{y,x}$ indexed by the set of edges $E$ and a non negative killing measure $\kappa$ on the set of vertices $X$, we can associate an energy (or Dirichlet form) $\mathcal{E}$, we will assume to be positive definite, which is a transience assumption. For any function $f$ on $X$, we have: $$\mathcal{E}(f;f)=\frac{1}{2}\sum_{x,y}C_{x,y}(f(x)-f(y))^{2}+\sum_x \kappa_x f^{2}(x).$$
There is a duality measure $\lambda$ defined by $\lambda_x=\sum_y C_{x,y} +\kappa_x$.
Let $G_{x,y}$ be the symmetric Green's function associated with $\mathcal{E}$. It is assumed that $\sum_x G_{x,x}\lambda_x$ is finite.\\
The symmetric associated continuous time Markov process can be obtained from the Markov chain defined by the transition matrix $P_{x,y}=\frac{C_{x,y}}{\lambda_y}$ by adding independent exponential holding times of mean $1$ before each jump. If $P$ is submarkovian, the chain is absorbed at a cemetery point $\Delta$. If $X$ is finite, the transition matrix is necessarily submarkovian.\\
The complex (respectively real) free field is the complex (real) Gaussian field on $X$ whose covariance function is $G$. We will denote it by $\varphi$ (respectively $\varphi_{\mathbb{R}}$).\\
We denote by $\mu$ the loop measure associated with this symmetric Markov process ($\mu$ can be also viewed as a shift invariant measure on based loops). We refer to \cite{stfl} for the general definition in terms of Markovian bridges, but let us mention that the measure of a non-trivial discrete loop is the product of the transition probabilities of the edges divided by the multiplicity of the loop. The measure on continuous time loops is then obtained by including exponential holding times, except for one point loops on which the holding time measure (which has infinite mass)  has density $\frac{e^{-t}}{t}$.

 The Poissonian loop ensemble $\mathcal{L}$ is the Poisson process of loops of intensity $\mu$. Recall that when $\mathcal{G}$ is finite, it can be sampled by the following algorithm (Cf: \cite{stfl}, \cite{chang}):\\
\begin{itemize}

\item First step: After choosing any ordering of $X$, define a set of based loops around each vertex. These are the based loops of the continuous Markov chain erased in the construction through Wilson's algorithm of the random spanning tree rooted in the cemetery point associated with the discrete time Markov chain.
\\
Recall that this construction goes as follows: We first build a loop erased chain $ \{ \eta_{1} \}$ starting from the first vertex and ending at the cemetery point $\Delta$, then, from the next vertex in $X - \{ \eta_{1} \}$ a loop-erased chain $ \eta_{2} $ ending in $\{ \eta_{1} \}\cup \{ \Delta \}$ and so on... The union of the  $ \eta_{i}$ is a spanning tree.\\
This set of erased based loops includes one point loops which are the holding times at points which are traversed only once.\\
\item Second step: Split these based loops into a collection of smaller based loops by dividing the local time at their base points according to independent Poisson-Dirichlet random variables.\\
\item  Finally, map this new set of based loops to their equivalence classes under the shift (this is the definition of loops).\\
\end{itemize}

This construction is related to the famous excursion point process introduced by K. It\^o in \cite{Ito}. Indeed, a similar construction can be done with linear Brownian motion or more generally one-dimensional diffusions (Cf \cite{Lupdif}). It can be also extended to non symmetric Markov chains (Cf \cite{chang}). 
\section{Occupation field and isomorphisms}
Given any vertex $x$ of the graph, denote by $\hat{\mathcal{L}}^x$ the total time spent in $x$ by the loops, normalized by $\lambda_x$.  $\hat{\mathcal{L}}$ is known as the occupation field of  $\mathcal{L}$.\\
Recall that as a Poisson process, $\mathcal{L}$ is infinitely divisible. We denote by $\mathcal{L}_\alpha$ the Poisson process of loops of intensity $ \alpha\mu$ and by $\hat{\mathcal{L}}_{\alpha}$ the associated occupation field. \\
It has been shown in $\cite{aop} $ (see also $\cite{stfl}$) that the fields $\hat{\mathcal{L}}$ ($\hat{\mathcal{L}}_{\frac{1}{2}}$) and $\frac{1}{2} \varphi^2$ ($\frac{1}{2} \varphi_{\mathbb{R}}^2$) have the same distribution. Note that this property extends naturally to one dimensional diffusions. Generalisations to dimensions 2 and 3 involve renormalization (Cf \cite{stfl} \cite{aop}).

We can use the second identity in law to give a simple proof of a result known as the \textit{generalized second Ray-Knight theorem} $(\cite{GRK} , \cite{SZ}, \cite{Sab} )$. Let $x_0$ be a point of $X$, and assume that $\kappa$ is supported by  $x_0$. Set $D=X-\lbrace x_0 \rbrace $.
Then it follows from the classical energy decomposition that $$\varphi_{\mathbb{R}}=\varphi_{\mathbb{R}}^D+\varphi_{\mathbb{R}}(x_0)$$ and $\varphi_{\mathbb{R}}^D$ (the real free field associated with the restriction of $\mathcal{E}$ to $D$) is independent of $\varphi_{\mathbb{R}}(x_0)$.\\
On the other hand, $$\hat{\mathcal{L}}_{\frac{1}{2}}=\hat{\mathcal{L}}^D_{\frac{1}{2}}+\hat{\mathcal{L}}^{(x_0)}_{\frac{1}{2}}$$ where $\hat{\mathcal{L}}^{(x_0)}_{\frac{1}{2}}$ denotes the occupation field of the set of loops of $\mathcal{L}_\frac{1}{2}$ hitting $x_0$ and $\hat{\mathcal{L}}^D_{\frac{1}{2}} $ denotes the occupation field of the set of loops of $\mathcal{L}_\frac{1}{2}$ contained in $D$. 

The two terms of the decomposition are clearly independent.      \\ 
 Moreover, given that its value at $x_0$ is $\rho$, the field $\hat{\mathcal{L}}^{(x_0)}_{\frac{1}{2}}$ has the same distribution as the occupation field $\hat{\gamma}_{\tau_{\rho}}$ of an independent copy of the Markov chain started at $x_0$ and stopped when the  local time at $x_0$ equals $\rho$.\\The identity in law which is valid between $\hat{\mathcal{L}}_{\frac{1}{2}}^D$ and $\frac{1}{2} (\varphi_{\mathbb{R}}^D)^2$ 
 as well as between $\hat{\mathcal{L}}_{\frac{1}{2}}$ and  $\frac{1}{2} {\varphi_{\mathbb{R}}}^2$ can be desintegrated taking $\hat{\mathcal{L}}_{\frac{1}{2}}^{x_0}=\frac{1}{2} \varphi_{\mathbb{R}}^2(x_0)=\rho$.  Noting finally that the sign $\eta$ of $\varphi_{\mathbb{R}}$ at $x_0$ is independent of the other variables we get that
 $$\frac{1}{2} {\varphi_{\mathbb{R}}}^2+\hat{\gamma}_{\tau_{\rho}}\stackrel{(d)}{=} \hat{\mathcal{L}}_{\frac{1}{2}}^D+\hat{\gamma}_{\tau_{\rho}}\stackrel{(d)}{=}\frac{1}{2} ({\varphi_{\mathbb{R}}^D}+\eta \sqrt{2\rho})^2 $$
 but we have also, by symmetry of ${\varphi_{\mathbb{R}}^D}$,

$ \frac{1}{2} ({\varphi_{\mathbb{R}}^D}+\eta \sqrt{2\rho})^2\stackrel{(d)}{=}\frac{1}{2} ({\varphi_{\mathbb{R}}^D}+\sqrt{2\rho})^2 \stackrel{(d)}{=}\frac{1}{2} ({\varphi_{\mathbb{R}}^D}-\sqrt{2\rho})^2.$ 
   so that finally we have proved:
  \begin{proposition} $$\frac{1}{2} {\varphi_{\mathbb{R}}}^2+\hat{\gamma}_{\tau_{\rho}}\stackrel{(d)}{=}\frac{1}{2} ({\varphi_{\mathbb{R}}^D}+\sqrt{2\rho})^2 $$\\
  \end{proposition} 
   Note that a natural coupling of the free field with the occupation field of the loop ensemble of intensity $\frac{1}{2}\mu$ has been recently given by T. Lupu \cite{Lup}, using loop clusters.\\

\section{Jumping numbers}
In what follows, we will assume for simplicity that $\mathcal{G}$ is finite.

Given any oriented edge $(x,y)$ of the graph, denote by $N_{x,y}(l)$ the total number of jumps made from $x$ to $y$ by the loop $l$ and by $N^{(\alpha)}_{x,y}$ the total number of jumps made from $x$ to $y$ by the loops of $\mathcal{L}_{\alpha}$.
\\ Let $Z$ be any Hermitian matrix indexed by pairs of vertices  such that 
$\forall x,y,\:0<\vert Z_{x,y}\vert \leq 1 $, and such that all but a finite set of entries indexed by $K\times K$, with $K$ finite, are equal to $1$. We denote $N^{(1)}$ by $N$.\\
The content of the following lemma appeared already in \cite{stfl}.
\begin{lemma}\label{toto} 

Denote by $P^Z_{x,y}$ the matrix $P_{x,y}Z_{x,y}$.\\
\begin{enumerate}
\item[i)]  We have: $$E(\prod_{x\neq y} Z_{x,y}^{N^{(\alpha)}_{x,y}}) =\left[\frac{\det(I-P^{Z})}{I-P}\right]^{-\alpha}.$$
\item[ii)] Moreover, for $\alpha=1$,  $$E(\prod_{x\neq y} Z_{x,y}^{N_{x,y}}) =E(e^{\sum_{x\neq y}(\frac{1}{2} C_{x,y} (Z_{x,y}-1)\varphi_x \bar{\varphi}_y)} ).$$
\end{enumerate}
\end{lemma} 
 \begin{proof}  (See also chapter 6 of \cite{stfl}) .\\
 i)  The left side can be expressed as 
$\exp(\alpha\int (\prod_{x,y} Z_{x,y}^{N_{x,y}(l)}-1)\mu(dl))$ and \\
\[ \int (\prod_{x,y} Z_{x,y}^{N_{x,y}(l)}-1)\mu(dl)=\int (\prod_{m=2}^{p(l)} Z_{\xi_{m-1},\xi_m} \mu(dl)-\mu(\{\text{non trivial loops}\})\] \\
which is equal to
$\sum_{n=1}^{\infty}\frac{1}{n}(\text{Tr}([P^Z]^n)
-\text{Tr}(P^n))$. The result follows from the identity $\log(\det)=\text{Tr}(\log)$. \\
ii) The right hand side equals $\frac{\det(I-P
)}{\det(I-P^Z)}$.
\end{proof}

\section{Eulerian networks}
We define a network to be a matrix $k$ with  $\mathbb{N}$-valued coefficient which vanishes on the diagonal and on entries $(x,y)$ such that $\{x,y\}$ is not an edge of the graph. We say that $k$ is Eulerian if $$ \sum_y k_{x,y}= \sum_y k_{y,x}$$ For any Eulerian network $k$, we define $k_x$ to be $\sum k_{x,y}=\sum k_{y,x}$.
The matrix $N^{(\alpha)}$ defines a random network which verifies the Eulerian property.\\

The distribution of the random network defined by $\mathcal{L}_{\alpha}$  is given in the following:
\begin{proposition}
\begin{enumerate}
\item[i)] For any Eulerian network $k$, $$ P(N^{(\alpha)}=k)=C\det(I-P)^{\alpha} \prod_{x,y} P_{x,y}^{k_{x,y}}$$ where $C$ is the coefficient of $\prod_{x,y} P_{x,y}^{k_{x,y}} $ in $\mathrm{Per}_{\alpha}(P(k_x, x\in X)) $. \\ Here $P(k_x, x\in X)$ is denoting the $(\vert k \vert , \vert k \vert)$ matrix obtained by repeating $k_x$ times each column of $P$ of index $x$  and then $k_y$ times each line of index $y$  (with  $\vert k \vert = \sum_{x} k_{x}$).\\
\item[ii)] For $\alpha=1$, there is a simpler expression:  For any Eulerian network $k$,$$  P(N=k)=\det(I-P)\frac{\prod_x {k_x}!}{\prod_{x,y} k_{x,y}!} \prod_{x,y} P_{x,y}^{k_{x,y}}.$$
\end{enumerate}
\end{proposition} 

\begin{proof}:
i) follows from the expansion of $\det(I-P^Z)^{\alpha}$ using a well known expansion formula for determinants powers using $\alpha$-permanents  (Cf \cite{VJ}, \cite{MR}).

ii) follows from i), but we will rather give two different alternative proofs.\\

\medskip 
Let $\mathfrak{N}$ be the additive semigroup of networks and $\mathfrak{E}$ be the additive semigroup of Eulerian networks. On one hand, note that
$$ E(\prod_{x,y} Z_{x,y}^{N_{x,y}})=\sum_{k\in \mathfrak{E} } P(N=k)\prod_{x,y} Z_{x,y}^{k_{x,y}}$$ 
On the other hand, from the previous lemma, we get\\

$ E(\prod_{x,y} Z_{x,y}^{N_{x,y}})=E(e^{\sum_{x,y}(\frac{1}{2} C_{x,y} (Z_{x,y}-1)\varphi_x \bar{\varphi}_y)} )\\
\\
= \frac{1}{(2\pi)^d \det(G)} \int e^{-\frac{1}{2}(\sum_{x}\lambda_x \varphi_x\bar{\varphi}_x -\sum_{(x,y)\in K\times K} C_{x,y} Z_{x,y}\varphi_x\bar{\varphi}_y)}\prod_x  \frac{1}{2i}d \varphi_x\wedge d\bar{\varphi}_x \\
\\
=\frac{1}{(2\pi)^d \det(G)} \int_0^\infty \int_0^{2\pi} e^{-\frac{1}{2}(\sum_{x}\lambda_x r_x^2 -\sum_{x,y} C_{x,y} Z_{x,y}r_x r_y e^{i(\theta_x-\theta_y)})}\prod_x r_x d r_xd{\theta}_x\\
\\
=\frac{1}{(2\pi)^d \det(G)} \int_0^\infty \int_0^{2\pi} e^{-\frac{1}{2}\sum_{x}\lambda_x r_x^2 }\sum_{n\in \mathfrak{N}}\prod_{x,y\in K} \frac{1}{n_{x,y}!}(C_{x,y}(\frac{1}{2}Z_{x,y}r_x r_y e^{i(\theta_x-\theta_y)})^{n_{x,y}} \prod_x r_x d r_xd{\theta}_x$\\

Integrating in the $\theta_x$ variables and using the definition of Eulerian networks, it equals
\\

$\frac{1}{\det(G)} \int_0^\infty  e^{-\frac{1}{2}\sum_{x}\lambda_x r_x^2 }\sum_{n\in \mathfrak{E}}\prod_{(x,y)\in K\times K} \frac{1}{n_{x,y}!}(\frac{1}{2}C_{x,y} Z_{x,y}r_x r_y )^{n_{x,y}} \prod_x r_x d r_x\\ \\
=\frac{1}{\det(G) \prod \lambda_x}\sum_{n\in \mathfrak{E}} \prod_{x \in K} n_x! \prod_{(x,y)\in K\times K}\frac{1}{n_{x,y}!}(\frac{C_{x,y}}{\lambda_x} Z_{x,y} )^{n_{x,y}}\\ \\
=\det(I-P) \sum_{n\in \mathfrak{E}} \prod_{x \in K} n_x! \prod_{(x,y)\in K\times K}\frac{1}{n_{x,y}!}(P_{x,y} Z_{x,y} )^{n_{x,y}}$.\\ 

We conclude the proof of the proposition by identifying the coefficients of $\prod_{x,y} Z_{x,y}^{k_{x,y}}$\\
Note that in the case of a space with two points $a$ and $b$, we see that the number of jumps $N_{a,b}$ follows a geometric distribution if $\alpha=1$. For general $\alpha$, we get a negative binomial distribution.\\

An alternative proof can be derived using Wilson's algorithm. It goes roughly as follows. We fix an order on the vertices: $X=\{x_1,\dots  ,x_n\}$  and denote by $\Lambda$ the set of $n$-uples of (possibly trivial) loops $\{l_1,\dots , l_n\}$ such that $l_i$ is rooted in $x_i$ and avoids $\{x_1,\dots  ,x_{i-1}\}$. Given an Eulerian network $k$, there are  $\frac{\prod_x {k_x}!}{\prod_{x,y} k_{x,y}!}$ different elements of $\Lambda$ inducing it. Indeed, $\frac{{k_x}!}{\prod_{x,y} k_{x,y}!}$ is the number of ways to order the oriented edges exiting from $x$. Once  such a choice has been made at very vertex, a $n$-uple of loops is determined: one starts with the first edge of the first vertex and then take the first edge of its endpoint, and so on until we come back to the first vertex and no exiting edge is left at this point. Then we move to the next vertex, etc... \\
Given any oriented spanning tree $T$ rooted in $\Delta$, all the elements of $\Lambda$ have the same probability $\prod_{x,y} P_{x,y}^{k_{x,y}}\prod_{(uv,)\in T}P_{u,v}$ to be (jointly with $T$) the output of Wilson's algorithm. We get the result by summing on all spanning trees.
\end{proof}\\
\section{Random homology}
Note that the additive semigroup of Eulerian networks is naturally mapped on the first homology group $H_1(\mathcal{G},\mathbb{Z})$ of the graph, which is defined as the quotient of the fundamental group by the subgroup of commutators. It is an Abelian group with $n=\vert E \vert -\vert X \vert +1$ generators. The homology class of the network $k$ is determined by the antisymmetric part $\widecheck{k}$ of the matrix $k$.\\The distribution of the induced random homology $\widecheck{N}^{(\alpha)}$  seems more difficult to compute explicitly. Note however that for any $j\in H_1(\mathcal{G},\mathbb{Z})$, $P(\widecheck{N}^{(\alpha)}=j)$ can be computed as a Fourier integral on the Jacobian torus of the graph $Jac(\mathcal{G})=H^1(\mathcal{G},\mathbb{R})/ H^1(\mathcal{G},\mathbb{Z})$. \\Here, following the approach of \cite{kosu}  we denote by $H^1(\mathcal{G},\mathbb{R})$ the space of harmonic one-forms , which in our context is the space of one-forms $\omega^{x,y}=-\omega^{y,x}$ such that $\sum_{y}C_{x,y}\omega^{x,y}=0$ for all $x\in X$ and by $H^1(\mathcal{G},\mathbb{Z})$ the space of harmonic one-forms $\omega$ such that for all discrete loop (or equivalently for all non backtracking discrete loop) $\gamma$ the holonomy $\omega(\gamma)$ is an integer.\\
More precisely, if we equip $H^1(\mathcal{G},\mathbb{R})$ with the scalar product defined by the set of conductances $C$: $$\Vert \omega\Vert^2=\sum_{x,y}C_{x,y}(\omega^{x,y})^2$$ and let $d\omega$ be the associated Lebesgue measure, for all $j \in H_1(\mathcal{G},\mathbb{Z})$, $$P(\widecheck{N}^{(\alpha)}=j)=\frac{1}{\vert \mathrm{Jac}(\mathcal{G}\vert}\int_{Jac(\mathcal{G})}E(e^{2\pi i\langle \widecheck{N}^{(\alpha)}-j,\omega \rangle} ) d\omega $$\\
$$=\frac{1}{\vert \mathrm{Jac}(\mathcal{G})\vert}\int_{\mathrm{Jac}(\mathcal{G})}\left[\frac{\det(I-P^{e^{2\pi i\omega}})}{\det(I-P)}\right]^{-\alpha}e^{-2\pi i\langle j,\omega \rangle}d\omega.$$
For $\alpha=1$, this expression can be written equivalently as $$=\frac{1}{\vert \mathrm{Jac}(\mathcal{G})\vert}\int_{\mathrm{Jac}(\mathcal{G})} E(e^{\sum_{x\neq y}(\frac{1}{2} C_{x,y} (e^{2\pi i \omega_{x,y}}-1)\varphi_x \bar{\varphi}_y)} ) e^{-2\pi i\langle j,\omega \rangle}d\omega$$
$$=\frac{1}{\vert \mathrm{Jac}(\mathcal{G})\vert}\int_{\mathrm{Jac}(\mathcal{G})} E(e^{\frac{1}{2} (\mathcal{E}-\mathcal{E}^{(2\pi i \omega)})(\varphi,\bar{\varphi}) )}e^{-2 \pi  i \langle j ,\omega \rangle} d\omega$$ where $\mathcal{E}^{(2\pi i \omega)}$ denotes the positive energy form defined by :$$\mathcal{E}^{(2\pi i \omega)}(f,g)=\frac{1}{2}\sum_{x,y}C_{x,y}(f(x)-e^{2\pi i \omega_{x,y}}f(y))(\bar{g}(x)-e^{-2\pi i \omega_{x,y}}\bar{g}(y))+\sum_x \kappa_x f^{2}(x)$$ This expession can also be written as
$$\frac{1}{\vert\mathrm{Jac}(\mathcal{G})\vert } \frac{1}{\det(G)}\int_{\mathrm{Jac}(\mathcal{G})} E(e^{-\frac{1}{2} \mathcal{E}^{(2\pi i \omega)}(\varphi,\bar{\varphi}) }e^{-2 \pi  i \langle j ,\omega \rangle} d\omega \dfrac{d\varphi \wedge d\bar{\varphi}}{2i}$$\\
This calculation suggests a coupling between the loop ensembles and the gauge field measure defined on $H^1(\mathcal{G},\mathbb{R})$ .
\\ Note finally that by adapting the results proved in \cite{kosu}, we see that the volume of the Jacobian torus $\vert \mathrm{Jac}(\mathcal{G})\vert$ can be expressed as the inverse of the square root of a number which can be computed in two ways:

a) as  $ \det (\Lambda) \prod _{e\in E} C_e$ with $\Lambda$ denoting the $(n,n)$ intersection matrix defined by any base $c_i$ of the  module $H_1(\mathcal{G},\mathbb{Z})$: $\Lambda_{i,j}=\langle c_i , c_j\rangle$. It is well known that such a base can be defined by any spanning tree: One considers a spanning tree $T$\ of the graph, and choose an
orientation on each of the $n$ remaining edges. This defines $n$ oriented
cycles on the graph and a system of $n$ generators for the fundamental group and for the homology group. (See \cite{Mass} or (\cite{Ser}) in a more general context).
 In this definition of $\Lambda$ we use the dual scalar product of the scalar product defined on the space of harmonic forms. It is induced by the scalar product on functions on the set of oriented edges $E^{o}$ by $\langle e, \pm e\rangle=\pm \frac{1}{C_e}$ and  $\langle c_i , c_j\rangle=0$ if $e \neq \pm e'$.\\
 
b) as  sum of the of the $C$-weights of the spanning trees of $\mathcal{G}$ (the weights are defined as the product of the conductances of their edges). See \cite{stfl} section 8-2 for a determinantal expression.

\section{Additional remarks}
\subsection{A determinant formula}
Recall that $N_x$ denotes $\sum_y N_{x,y}$. Lemma \ref{toto} can be stated in a more general form (cf \cite{stfl} (6-4)).
$$E(\prod_{x\neq y} Z_{x,y}^{N_{x,y}}\prod_{x} Z_{x,x}^{-(N_{x}+1)}) =E(e^{\sum_{x\neq y}(\frac{1}{2} C_{x,y} (Z_{x,y}-1)\varphi_x \bar{\varphi}_y)+\sum_{x}(\frac{1}{2} \lambda_x (1-Z_{x,x})\varphi_x \bar{\varphi}_x)} ).$$
A consequence is that is that for any set $(x_i,y_i)$ of distinct oriented edges, and any set $z_l$  of distinct vertices,
$$E(\prod_i N_{(x_i,y_i)}\prod_l (N_{z_l}+1))=E(\prod_i {\varphi_{x_i }\bar{\varphi}_{y_i}}C_{(x_i,y_i)} \prod_l {\lambda_{z_l } \varphi_{z_l }\bar{\varphi}_{z_l}} ).$$
In particular, if $X$ is assumed to be finite, if $[D_{N}]_{(x,y)}=0$ for $x\neq y$ and $[D_{N}]_{(x,x)}=1+N_{x}$, for all $\chi\geq \lambda$, then
$$E(\det( M_{\chi}D_{N}-N))=\det(M_{\varphi}(M_\chi-C)M_{\bar{\varphi}}) =\det(M_\chi-C) \text{Per}(G).$$\\
Note that $D_{N}-N$ is a Markov generator.

\subsection{An application of the BEST theorem.}

The BEST theorem (Cf \cite{Stan}) determines the measure induced on Eulerian networks by the restriction of $\mu$ to non trivial loops. If $k$ is a Eulerian network, let $\tilde{k}$ be the oriented Eulerian graph associated with it. Its set of vertices is $X$ and it has $k_{x,y}$ oriented edges from $x$ to $y$.  Let$\vert k \vert = \sum_{x} k_{x}$ be the total number of edges in $\tilde{k}$.
Note that all pointed loops inducing the same network $k \in \mathfrak{E}$ have the same measure $\frac{1}{\vert k \vert}\prod_{x,y} P_{x,y}^{k_{x,y}} $ and that there are $\prod_{x,y}k_{x,y}!$ rooted Eulerian tours of $\tilde{k}$, i.e. directed rooted closed paths visiting each edge of $\tilde{k}$ exactly once, inducing each of them. It follows that the $\mu$-measure of $k$ is given by $\frac{N(k)}{ \prod_{x,y}k_{x,y}!} \frac{1}{\vert k \vert}\prod_{x,y} P_{x,y}^{k_{x,y}} $ where $N(k)$ is the number of Eulerian tours of $\tilde{k}$. It is given by the BEST theorem:$$N(k)=\vert k \vert \,\tau(\tilde{k})\prod_x (k_x-1)!$$ where $\tau(\tilde{k})$ is the number of arborescences of $\tilde{k}$ (which can be obtained by the matrix-tree theorem (Cf $\cite{Stan}$ )) and the factor $\vert k \vert$ takes into account the choice of the first oriented edge in the Eulerian tour. Hence, 
$$\mu(k)=\tau(\tilde{k})\prod _x(k_x-1)!\prod_{x,y} \frac{P_{x,y}^{k_{x,y}}}{k_{x,y}!}$$ for $k$ non zero. We know already that  the total $\mu$ measure of non zero networks is $- \log(\det(I-P))$.\\Our probability on Eulerian networks is therefore a sum of the convolutions powers of this measure with Poissonnian weights. I do not know of any combinatorial proof of that fact.

\subsection{Open questions}
Given any pair $(A,B)$ of disjoint subsets of $X$, a network $k$ defines an integer $N_{A,B}(k)$ which is the value of the maximal flow from $A$ to $B$. It is also the minimal number of cuts (on the associated digraph $\tilde{k}$ ) which separate $A$ and $B$. See for exemple \cite{Behz} . Percolation from $A$ to $B$ means that $N_{A,B}(k)$ is non-zero. For the random networks, these random variables are of special interest as are the resolvent, the Poisson kernels and the spectrum of the random Markov generator $D_{N}-N$. Their asymptotic properties should be investigated in connection with the percolation problem for loop clusters (Cf: \cite{ljlm},\cite{ch},\cite{Lup},\cite{chsa},\cite{Lup2}).\\
The coupling with gauge fields was already mentionned in the homology section.

\bigskip

\noindent

  Universit\'e Paris-Sud, B\^atiment 425, Orsay, France.

\bigskip
   yves.lejan@math.u-psud.fr

\end{document}